\documentclass[a4paper,11pt]{amsart}
\usepackage[
  top=3cm,
  bottom=3cm,
  left=3cm,
  right=3cm,
  headheight=14pt,
  headsep=1cm,
  footskip=1cm
]{geometry}
\usepackage{amsmath, epsfig, verbatim}
\usepackage{amsfonts}
\usepackage{amsthm}
\usepackage{mathtools}
\usepackage{amsmath}
\usepackage{amssymb} 
\usepackage{graphics,graphicx}
\usepackage{epstopdf}
\usepackage{listings}
\usepackage{xcolor}
\usepackage{tikz}
\usepackage{subcaption}
\usepackage{pgfplots}
\usepackage{caption}
\usepackage{booktabs}
\usepackage{tabularx}
\usepackage{hyperref}   
\usepackage{doi}
\usepackage{enumitem}
\usepackage[numbers,sort&compress]{natbib}
\usetikzlibrary{calc}
\allowdisplaybreaks
\definecolor{codegray}{rgb}{0.95,0.95,0.95}
\definecolor{pykeyword}{rgb}{0.13,0.13,1}
\definecolor{pystring}{rgb}{0.58,0,0.82}

\lstdefinestyle{pythonstyle}{
    backgroundcolor=\color{codegray},
    language=Python,
    basicstyle=\ttfamily\small,
    keywordstyle=\color{pykeyword}\bfseries,
    stringstyle=\color{pystring},
    commentstyle=\color{gray},
    showstringspaces=false,
    numbers=left,
    numberstyle=\tiny,
    frame=single,
    breaklines=true,
    tabsize=4,
}

\numberwithin{equation}{section}
\theoremstyle{plain}
\newtheorem{theorem}{Theorem}
\newtheorem{defn}[theorem]{Definition}
\newtheorem{coro}[theorem]{Corollary}
\newtheorem{lemma}[theorem]{Lemma}
\newtheorem{prop}[theorem]{Proposition}

\begin{document}
\title{Connected Mutual-Visibility in Graphs} 
\author{Tonny K B}
\address{Tonny K B, Department of Mathematics, College of Engineering Trivandrum, Thiruvananthapuram, Kerala, India, 695016.}
\email{tonnykbd@cet.ac.in}
\author{Shikhi M}
\address{Shikhi M, Department of Mathematics, College of Engineering Trivandrum, Thiruvananthapuram, Kerala, India, 695016.}
\email{shikhim@cet.ac.in}

\begin{abstract}
A set $S$ of vertices of a graph $G$ is a connected mutual-visibility
set if every two vertices of $S$ are joined by a shortest path whose
internal vertices lie outside $S$, and the subgraph induced by $S$ is
connected. We introduce the connected mutual-visibility number
$\mu_c(G)$, defined as the maximum cardinality of such a set, and
investigate its structural and algorithmic properties. We establish
fundamental bounds, derive Nordhaus--Gaddum type inequalities, and
characterise the graphs attaining the minimum and maximum possible
values. For regular $(d,2,-\delta)$-graphs, we derive general
bounds on $\mu_c(G)$ and determine its exact value for the two cubic
graphs of defect $2$. We further show that $\mu_c(G)$ is determined locally
by the block structure of $G$, namely, it is equal to the maximum of
the corresponding values over the blocks of $G$. Finally, we present a
polynomial-time algorithm for recognising connected mutual-visibility
sets and prove that the associated decision problem is
$\mathsf{NP}$-complete, even for connected bipartite graphs of diameter
at most $4$.
\end{abstract}
\maketitle

\section{Introduction}

Let $S\subseteq V(G)$. Two vertices $u,v\in S$ are said to be
\emph{$S$-visible} \cite{MV_3} if there exists a $(u,v)$-geodesic whose internal
vertices do not belong to $S$. The set $S$ is called a
\emph{mutual-visibility set} of $G$ if every two distinct vertices of
$S$ are $S$-visible. The maximum cardinality of a mutual-visibility
set \cite{Stefano} of $G$ is called the \emph{mutual-visibility number} of $G$ and is
denoted by $\mu(G)$. A mutual-visibility set of cardinality $\mu(G)$
is called a \emph{$\mu$-set} of $G$. The concept of mutual visibility has attracted considerable attention
in recent years, leading to several variants and extensions as well as
investigations of its structural, algorithmic, and extremal
properties~\cite{DetMag2025, MV_3}.

Mutual visibility has applications in multi-agent robotic systems,
where visibility among robots is essential for coordinated operation,
formation control, surveillance, and exploration of unknown or dynamic
environments
\cite{robotics1,robotics2,robotics3,robotics4,robotics5,robotics6,robotics7}.
It also admits a natural interpretation in communication networks,
where the selected vertices represent cooperating agents and
geodesics whose internal vertices lie outside the selected set
correspond to efficient communication routes.  In this setting, the selected vertices represent
cooperating agents, while geodesics whose internal vertices lie
outside the selected set model efficient communication through the
surrounding network.

In many practical situations, however, the selected agents are also
required to form a connected communication structure. For example,
they may represent a trusted team whose members exchange confidential
or mission-critical information internally while still interacting
efficiently with one another through the ambient network. The internal
communication route need not be shortest, since all intermediate
vertices belong to the trusted team. On the other hand, communication
through vertices outside the team may involve unreliable, costly, or
otherwise unfavourable conditions, making shortest routes preferable.
These considerations naturally motivate imposing connectivity on the
selected mutual-visibility set.

Motivated by these considerations, we introduce the notion of
connected mutual visibility. A connected mutual-visibility set
is a mutual-visibility set whose induced subgraph is connected. Thus, the parameter combines two
fundamental requirements: the existence of shortest external routes
between every pair of selected vertices and the existence of an
internal connected communication network among them.

The study of connected mutual visibility reveals several phenomena
that differ significantly from those of the classical parameter. In
particular, connectivity imposes additional structural restrictions on
maximum mutual-visibility sets, leading to new extremal
characterisations, exact values for several graph classes and graph
operations, and new complexity questions.

In this paper, we establish fundamental properties of connected
mutual visibility. We derive basic bounds and characterise graphs
attaining the extremal values of $\mu_c(G)$. We determine the
connected mutual-visibility number for several important graph
classes, including geodetic graphs, joins,
block graphs, and cactus graphs. We also establish
Nordhaus--Gaddum-type inequalities, investigate regular defect
graphs of diameter $2$, and prove that recognising connected
mutual-visibility sets is polynomial-time solvable, whereas the
associated optimisation decision problem is
$\mathsf{NP}$-complete even for bipartite graphs of diameter at most
$4$.

\section{Notations and preliminaries}
Let $G=(V,E)$ be a finite, simple, undirected graph. We denote its vertex
set and edge set by $V(G)$ and $E(G)$, respectively. Unless stated otherwise,
all graphs considered in this paper are connected. A graph is \emph{geodetic} if every pair of vertices is
joined by a unique shortest path. A sequence of distinct vertices $(u_0,u_1,\ldots,u_n)$ is called a
$(u_0,u_n)$\emph{-path} in a graph $G$ if
$u_iu_{i+1}\in E(G)$ for every $i\in{0,1,\ldots,n-1}$. A
\emph{cycle} is obtained from a $(u_0,u_n)$-path by adding the edge
$u_0u_n$. The cycle on $n$ vertices is denoted by $C_n$. A graph $G$
is \emph{triangle-free} if it contains no cycle of length $3$. The \emph{complement} of a graph $G$, denoted by $\overline{G}$, is the
graph with vertex set $V(G)$ in which two distinct vertices are
adjacent if and only if they are non-adjacent in $G$. For
$S\subseteq V(G)$, the \emph{subgraph of $G$ induced by $S$}, denoted
by $G[S]$, is the graph with vertex set $S$ in which two distinct
vertices are adjacent if and only if they are adjacent in $G$. A subgraph $H$ of a graph $G$ is called \emph{convex} if every shortest
path in $G$ between two vertices of $H$ is contained in $H$.

Let $X$ be a set, and let $|X|$ denote its cardinality. For a vertex
$v\in V(G)$, let $N_G(v)$ denote the set of vertices adjacent to $v$ in $G$,
and let $d_G(v)=|N_G(v)|$ denote the degree of $v$ in $G$. A graph $G$
is \emph{$d$-regular} if every vertex of $G$ has degree $d$. A \emph{cubic graph} is a $3$-regular graph. The \emph{girth} of a graph $G$, denoted by $g(G)$, is the length of a
shortest cycle in $G$. If $G$ is acyclic, then $g(G)=\infty$. A subset $C\subseteq V(G)$ is a \emph{clique} of $G$ if every two
distinct vertices of $C$ are adjacent. The maximum cardinality of a
clique in $G$ is called the \emph{clique number} of $G$ and is denoted
by $\omega(G)$. A \emph{block} of a graph $G$ is a maximal connected subgraph of $G$
having no cut-vertex. A graph $G$ is called a \emph{block graph} if
every block of $G$ is complete. A graph $G$ is called a \emph{cactus graph} if any two cycles of $G$
have at most one vertex in common. For two disjoint sets $S,T\subseteq V(G)$, let $e(S,T)$ denote the
number of edges of $G$ with one endpoint in $S$ and the other in $T$.

For integers $d\ge2$ and $k\ge1$, the \emph{Moore bound} is

$$
M(d,k)=1+d\sum_{i=0}^{k-1}(d-1)^i.
$$

A \emph{$(d,k,-\delta)$-graph} is a $d$-regular graph of diameter $k$
and order $M(d,k)-\delta$, where $\delta\ge0$ is called the
\emph{defect} of the graph. In particular,
$M(d,2)=d^2+1$, and therefore a $(d,2,-\delta)$-graph has order
$d^2+1-\delta$. Graphs of defect $0$ attain the Moore bound and are called
\emph{Moore graphs}. For $k=2$ and $d\ge2$, Moore graphs are known to
exist only for $d=2,3,7$, and possibly $57$
\cite{hoffman_moore}. The corresponding graphs are $C_5$, the
Petersen graph, and the Hoffman--Singleton graph, while the existence
of a Moore graph of degree $57$ remains an open problem. For
$k\ge3$ and $d=2$, Moore graphs are precisely the cycles of length
$2k+1$. Moreover, for $k\ge3$ and $d\ge3$, Moore graphs do not
exist~\cite{BanIto1973,Dam1973}. A comprehensive account of Moore
graphs is given in~\cite{MillerSiran2013}.

Let $G$ and $H$ be graphs. The \emph{corona product} of $G$ and $H$,
denoted by $G\odot H$, is the graph obtained from one copy of $G$ and,
for each vertex $v\in V(G)$, a copy $H_v$ of $H$ by joining $v$ to
every vertex of $H_v$. If $G$ and $H$ are vertex-disjoint, then the
\emph{join} of $G$ and $H$, denoted by $G\vee H$, is the graph
obtained from their disjoint union by joining every vertex of $G$ to
every vertex of $H$.

\section{Connected mutual-visibility}
\begin{defn}
Let $G$ be a graph and let $S\subseteq V(G)$. The set $S$ is called a
\emph{connected mutual-visibility set} of $G$ if $S$ is a mutual-visibility
set of $G$ and the induced subgraph $G[S]$ is connected. The maximum
cardinality of a connected mutual-visibility set of $G$ is called the
\emph{connected mutual-visibility number} of $G$ and is denoted by
$\mu_c(G)$. A connected mutual-visibility set of cardinality $\mu_c(G)$ is
called a $\mu_c$-set of $G$.
\end{defn}
The vertex set of every clique in a graph is a connected
mutual-visibility set. Consequently,
for every graph $G$, $2 \leq \omega(G)\leq \mu_c(G)\leq \mu(G)$, where $\omega(G)$ denotes the clique number of $G$.

\begin{prop}\label{P16.prop1}
Let $G$ be a graph of order $n\geq 2$. Then $\mu_c(G)=n$ if and only if
$G\cong K_n$.
\end{prop}
\begin{proof}
If $G\cong K_n$, then $V(G)$ is a connected mutual-visibility set, and hence $\mu_c(G)=n$.

Conversely, suppose that $\mu_c(G)=n$. Since every connected
mutual-visibility set is a mutual-visibility set, we have
$\mu_c(G)\leq \mu(G)$. Hence $
n=\mu_c(G)\leq \mu(G)\leq n$, 
and so $\mu(G)=n$. By Lemma~4.1(3) of~\cite{Stefano}, it follows that
$G\cong K_n$.
\end{proof}
\begin{theorem}\label{P16.th3}
Let $G$ be a graph of order at least $2$. Then
$\mu_c(G)=2$ if and only if $g(G)\geq 5$.
\end{theorem}

\begin{proof}
Suppose that $\mu_c(G)=2$. If $G$ contains a triangle, then the
vertex set of this triangle is a connected mutual-visibility set of
cardinality $3$, a contradiction. Hence $G$ is triangle-free. If
$(v_1,v_2,v_3,v_4,v_1)$ is a $4$-cycle, then $v_1$ and $v_3$ are
non-adjacent. The set $S=\{v_1,v_2,v_3\}$ induces a connected subgraph,
and $(v_1,v_4,v_3)$ is a shortest $(v_1,v_3)$-path whose internal vertex
lies outside $S$. Thus $S$ is a connected mutual-visibility set of
cardinality $3$, again a contradiction. Therefore, $G$ contains neither
a triangle nor a $4$-cycle, and so $g(G)\geq 5$.

Conversely, suppose that $g(G)\geq 5$, and let $S\subseteq V(G)$ induce
a connected subgraph with $|S|\geq 3$. Since $G$ is triangle-free,
$G[S]$ contains an induced path $(u,v,w)$. If $u$ and $w$ had a common
neighbour different from $v$, then $G$ would contain a $4$-cycle. Hence
$v$ is their unique common neighbour, and so $(u,v,w)$ is the unique
shortest $(u,w)$-path. Since its internal vertex $v$ belongs to $S$, the
vertices $u$ and $w$ are not $S$-visible. Thus no connected
mutual-visibility set has cardinality at least $3$. Since $G$ is
connected and has order at least $2$, we have $2\leq \mu_c(G)$.
Therefore, $\mu_c(G)=2$.
\end{proof}
Let $P$ and $H$ denote the Petersen graph and the Hoffman--Singleton graph,
respectively. Both graphs have girth $5$. Hence
Theorem~\ref{P16.th3} gives $\mu_c(P)=\mu_c(H)=2$. This is in clear
contrast with their classical mutual-visibility numbers, namely
$\mu(P)=6$~\cite{GP_1} and $\mu(H)=20$~\cite{TonShi_dif2026}.
\begin{theorem}\label{P16.th9}
Let $G$ be a graph of order at least $2$, and let $S$ be a connected mutual-visibility set
with $|S|\ge2$. If $T$ is a spanning tree of $G[S]$, then
\[
|S|-2
\le
\sum_{x\in S}\binom{d_T(x)}{2}
\le
c_3(G)+c_4(G),
\]
where $c_i(G)$ denotes the number of cycles of length $i$ in $G$.
Consequently,
\[
\mu_c(G)\le c_3(G)+c_4(G)+2.
\]
\end{theorem}
\begin{proof}
Since $T$ is a tree of order at least $2$, every vertex of $T$ has
positive degree. Hence
\[
\binom{d_T(x)}{2}\ge d_T(x)-1
\quad\text{for every }x\in S.
\]
Summing over all vertices of $T$ and using the
handshaking lemma gives
\[
\sum_{x\in S}\binom{d_T(x)}{2}
\ge
\sum_{x\in S}(d_T(x)-1)
=
2|E(T)|-|S|
=
|S|-2.
\]

For every non-leaf vertex $x$ of $T$ and every unordered pair
$\{u,v\}$ of distinct neighbours of $x$ in $T$, consider the
$2$-path $(u,x,v)$. The number of such $2$-paths is $\sum_{x\in S}\binom{d_T(x)}{2}$,
since leaves contribute zero.

If $u$ and $v$ are adjacent in $G$, then $\{u,x,v\}$ induces a
triangle of $G$. Suppose therefore that $u$ and $v$ are non-adjacent.
Since $ux,vx\in E(G)$, we have $d_G(u,v)=2$. Since $S$ is a mutual-visibility set, there exists a $(u,v)$-geodesic
of length $2$ whose unique internal vertex lies outside $S$. Hence there is a vertex
$y\in V(G)\setminus S$ adjacent to both $u$ and $v$, and
$(u,x,v,y,u)$ is a $4$-cycle of $G$.

We claim that distinct $2$-paths of $T$ produce distinct triangles or
$4$-cycles in $G$.

Assume that two distinct $2$-paths of $T$ produce the same triangle.
The two $2$-paths cannot have the same middle vertex, for otherwise
they would have the same end vertices and hence be identical.
Thus their middle vertices are distinct. Each path contributes the two edges of the triangle incident with its
middle vertex. Consequently, their union is the entire triangle,
contradicting the fact that $T$ is a tree.

Now suppose that two distinct $2$-paths,
$(u,x,v)$ and $(u',x',v')$, produce the same $4$-cycle. Let $y$ and
$y'$ denote the corresponding vertices outside $S$. Since each
constructed $4$-cycle contains exactly one vertex outside $S$, we have
$y=y'$. In this $4$-cycle, $x$ is the unique vertex opposite to $y$.
Hence $x=x'$. Consequently,
$\{u,v\}=\{u',v'\}$, and the two $2$-paths coincide, a contradiction.
Therefore, distinct $2$-paths of $T$ produce distinct triangles or
$4$-cycles in $G$.

It follows that
\[
\sum_{x\in S}\binom{d_T(x)}{2}
\le
c_3(G)+c_4(G).
\]

If $S$ is a $\mu_c$-set of $G$, then $|S|=\mu_c(G)$, and the lower
bound established above yields
\[
\mu_c(G)\le c_3(G)+c_4(G)+2.
\]
This completes the proof.
\end{proof}
\begin{theorem}\label{P16.th1}
If $G$ is a geodetic graph, then $\mu_c(G)=\omega(G)$.
\end{theorem}
\begin{proof}
Since the vertex set of every clique is a connected mutual-visibility set,
we have $\omega(G)\leq \mu_c(G)$.

Let $S$ be a connected mutual-visibility set of $G$. We claim that
$G[S]$ is complete. Suppose not. Then there exist non-adjacent vertices
$x,y\in S$. Since $G[S]$ is connected, let
$P=(x=x_0,x_1,\ldots,x_\ell=y)$ be a shortest $(x,y)$-path in $G[S]$.
As $x$ and $y$ are non-adjacent, $\ell\geq 2$. Moreover,
$x_0x_2\notin E(G)$, since otherwise $(x_0,x_2,\ldots,x_\ell)$ would
be a shorter $(x,y)$-path in $G[S]$. Thus $(x_0,x_1,x_2)$ is a path of length $2$ in $G$, and $d_G(x_0,x_2)=2$. Since $G$ is geodetic,
$(x_0,x_1,x_2)$ is the unique shortest $(x_0,x_2)$-path in $G$.
However, its internal vertex $x_1$ belongs to $S$. Therefore, $x_0$
and $x_2$ are not $S$-visible, contradicting the fact that $S$ is a
mutual-visibility set. It follows that $G[S]$ is complete, and hence
$|S|\leq \omega(G)$.

Since this holds for every connected mutual-visibility set $S$, we obtain
$\mu_c(G)\leq \omega(G)$. Combining the two inequalities gives
$\mu_c(G)=\omega(G)$.
\end{proof}
Since every tree is geodetic and every tree of order at least $2$ has clique
number $2$, Theorem~\ref{P16.th1} gives the following result.
\begin{coro}\label{P16.cor1}
If $T$ is a tree of order at least $2$, then $\mu_c(T)=2$. Consequently, $\mu_c(T)=\mu(T)$ if and only if $T$ is a path.
\end{coro}
\begin{proof}
    Since any tree $T$ is geodetic, $\mu_c(T)=2$. Since $\mu(T)$ equals the number of leaves of $T$, $\mu_c(T)=\mu(T)$ if and only if $T$ has exactly two leaves, which holds if and only if $T$ is a path.
\end{proof}
Nordhaus--Gaddum type inequalities relate the value of a graph invariant on a graph and its complement. Motivated by the classical Nordhaus--Gaddum bounds for the chromatic number, we investigate analogous sum and product bounds for the connected mutual-visibility number.

\begin{theorem}
Let $G$ be a graph of order $n\geq 4$. Then
$\mu_c(G)+\mu_c(\overline{G})\leq 2n-3$ and $
\mu_c(G)\mu_c(\overline{G})\leq (n-1)(n-2)$.
Moreover, both bounds are sharp.
\end{theorem}
\begin{proof}
If $\mu_c(G)=n$, then $G\cong K_n$, and hence
$\mu_c(\overline{G})=1$; similarly, if
$\mu_c(\overline{G})=n$, then $\mu_c(G)=1$. In either case, the desired
inequalities are immediate. Hence we may assume that $ 2\le \mu_c(G),\ \mu_c(\overline{G})\le n-1$. We next show that the case
$ \mu_c(G) = \mu_c(\overline{G})=n-1$ cannot occur.
Suppose that
$S=V(G)\setminus\{x\}$ and $T=V(G)\setminus\{y\}$ are connected
mutual-visibility sets of $G$ and $\overline{G}$, respectively. Since $x$ is the only vertex outside $S$, every non-adjacent pair
$u,v\in S$ is joined by the geodesic $(u,x,v)$. Hence
\begin{equation}\label{eq:NG1}
uv\notin E(G),\quad u,v\neq x
\quad\Longrightarrow\quad
ux,vx\in E(G).
\end{equation}
Similarly, since $y$ is the only vertex outside $T$, every edge
$uv$ of $G[T]$ satisfies
\begin{equation}\label{eq:NG2}
uv\in E(G),\quad u,v\neq y
\quad\Longrightarrow\quad
uy,vy\notin E(G).
\end{equation}

\noindent\textbf{Case 1.} Suppose that $x=y$. Since $G[S]$ is connected
and $|S|=n-1\ge3$, it contains an edge $uv$. By~\eqref{eq:NG2},
$xu,xv\notin E(G)$. On the other hand,
$\overline{G}[T]$ is connected, so $u$ has a neighbour
$w$ in $\overline{G}[T]$. Thus $uw\notin E(G)$, and
\eqref{eq:NG1} gives $xu,xw\in E(G)$, contradicting
$xu\notin E(G)$.

\noindent\textbf{Case 2.} Suppose that $x\neq y$, and let
$z\in V(G)\setminus\{x,y\}$. If $xz\in E(G)$, then
\eqref{eq:NG2} gives $xy,yz\notin E(G)$. Since
$yz\notin E(G)$ and $y,z\neq x$, \eqref{eq:NG1} yields
$xy\in E(G)$, a contradiction. Hence
$xz\notin E(G)$ for every
$z\in V(G)\setminus\{x,y\}$.

It follows from~\eqref{eq:NG1} that
$yz\in E(G)$ for every
$z\in V(G)\setminus\{x,y\}$; otherwise,
$yz\notin E(G)$ would imply $xz\in E(G)$. Since $n\ge4$, choose
distinct vertices $z,w\in V(G)\setminus\{x,y\}$. If
$zw\in E(G)$, then \eqref{eq:NG2} gives
$yz\notin E(G)$, a contradiction. Hence
$zw\notin E(G)$, and \eqref{eq:NG1} yields
$xz\in E(G)$, again a contradiction.

Therefore, at least one of $\mu_c(G)$ and $\mu_c(\overline{G})$ is at
most $n-2$. Hence
$\mu_c(G)+\mu_c(\overline{G})\le (n-1)+(n-2)=2n-3$ and
$\mu_c(G)\mu_c(\overline{G})\le (n-1)(n-2)$.

To prove sharpness, let $G=K_{n-2}\cup2K_1$. Since every connected
induced subgraph of $G$ is contained in one of its components,
Theorem~\ref{P16.th7} and Proposition~\ref{P16.prop1} yield
$\mu_c(G)=n-2$. Moreover,
$\overline{G}\cong K_2\vee\overline{K_{n-2}}$, and
Theorem~\ref{P16.th6} gives
$\mu_c(\overline{G})=n-1$. Hence $ \mu_c(G)+\mu_c(\overline{G})=2n-3$ and $
\mu_c(G)\mu_c(\overline{G})=(n-1)(n-2)$, 
showing that both bounds are sharp.
\end{proof}
\begin{lemma}\label{P16.lem1}
Let $H$ be a convex subgraph of a graph $G$. Then $\mu_c(H)\leq\mu_c(G)$.
\end{lemma}
\begin{proof}
Let $S$ be a $\mu_c$-set of $H$. Since $H$ is convex, every geodesic in
$H$ is also a geodesic in $G$. Hence $S$ is a mutual-visibility set of
$G$. Moreover, $H[S]$ is connected and is a subgraph of $G[S]$.
Therefore, $G[S]$ is connected, and thus $S$ is a connected
mutual-visibility set of $G$. Consequently,
$\mu_c(H)=|S|\le\mu_c(G)$.
\end{proof}
\begin{theorem}\label{P16.th7}
Let $G$ be a graph of order at least $2$, and let
$\mathcal B(G)$ denote the collection of blocks of $G$. Then
\[
\mu_c(G)=\max_{B\in\mathcal B(G)}\mu_c(B).
\]
\end{theorem}
\begin{proof}
Indeed, suppose that there is a path between two vertices of a block
$B$ which leaves $B$. Then the path contains a subpath whose distinct
endvertices belong to $B$ and whose internal vertices lie outside
$B$. The union of $B$ and this subpath has no cut-vertex, and hence
is contained in a block properly containing $B$, contradicting the
maximality of $B$. Therefore, by Lemma~\ref{P16.lem1},
$\mu_c(B)\le\mu_c(G)$ for every $B\in\mathcal B(G)$. Hence,
\[
\max_{B\in\mathcal B(G)}\mu_c(B)\le\mu_c(G).
\]

Let $S$ be a $\mu_c$-set of $G$. We claim that $S$ is contained in a
single block of $G$. Suppose otherwise. Then there exist
$u,v\in S$ and a cut-vertex $x$ of $G$ such that $u$ and $v$ lie in
distinct components of $G-x$. Since $G[S]$ is connected, every
$(u,v)$-path in $G[S]$ contains $x$, and hence $x\in S$. Moreover,
every $(u,v)$-path in $G$ contains $x$, so every $(u,v)$-geodesic has
the internal vertex $x\in S$. Thus $u$ and $v$ are not mutually visible
with respect to $S$, a contradiction.

Hence $S\subseteq V(D)$ for some block $D$ of $G$. Since $D$ is convex,
$S$ is a mutual-visibility set of $D$. Moreover, $D$ is induced, so
$D[S]=G[S]$ is connected. Therefore, $S$ is a connected
mutual-visibility set of $D$. Consequently,
\[
\mu_c(G)=|S|\le\mu_c(D)\le
\max_{B\in\mathcal B(G)}\mu_c(B),
\]
and the proof is complete.
\end{proof}
\section{Graphs of diameter two}
The diameter-$2$ case deserves separate consideration because the
mutual-visibility condition becomes local. Adjacent vertices are
automatically visible, while the visibility of two non-adjacent
vertices depends only on their common neighbours. Thus, connected
mutual visibility can be studied through the structure of the induced
subgraph together with common-neighbour information. The following
characterisation will be used repeatedly in this section.
\begin{lemma}[{\cite[Lem.~1]{TonShi_dif2026}}]\label{P6.lem1}
Let $G$ be a graph of diameter $2$. A subset $S \subseteq V(G)$ is a mutual-visibility set of $G$ if and only if, for each non-adjacent pair $\{u,v\}\subseteq S$, there is a vertex $w\in V(G)\setminus S$ that is adjacent to both $u$ and $v$.
\end{lemma}
\begin{theorem}\label{P16.th2}
Let $G=K_{n_1,\ldots,n_r}$ be a complete $r$-partite graph, where
$r\geq 2$, $1\leq n_1\leq \cdots \leq n_r$, and
$n=\sum_{i=1}^{r}n_i$. Then
\[
\mu_c(G)=
\begin{cases}
n, & n_r=1,\\[2mm]
2, & r=2 \text{ and } n_1=1<n_2,\\[2mm]
n-1, & n_r\geq 2,\ n_1\leq 2,\text{ and either } r\geq 3
       \text{ or } n_1=2,\\[2mm]
n-2, & n_1\geq 3.
\end{cases}
\]
\end{theorem}

\begin{proof}
Let $V_1,\ldots,V_r$ be the partite sets of $G$, where $|V_i|=n_i$ for
each $i\in\{1,\ldots,r\}$. If $n_r=1$, then every partite set has cardinality $1$, and hence
$G\cong K_n$. Therefore, Proposition~\ref{P16.prop1} gives
$\mu_c(G)=n$.

Suppose next that $r=2$ and $n_1=1<n_2$. Then $G$ is a star. Let
$V_1=\{x\}$. If a connected mutual-visibility set $S$ contains two
vertices $u,v\in V_2$, then it must also contain $x$, since $G[S]$ is
connected. However, $x$ is the unique common neighbour of $u$ and $v$.
Thus $(N_G(u)\cap N_G(v))\setminus S=\emptyset$, contradicting
Lemma~\ref{P6.lem1}. Hence $|S|\leq 2$. Since the vertex set of every
edge is a connected mutual-visibility set, it follows that
$\mu_c(G)=2$.

Now suppose that $n_r\geq 2$, $n_1\leq 2$, and either $r\geq 3$
or $n_1=2$. Choose $x\in V_1$ and let
$S=V(G)\setminus\{x\}$. Under these assumptions, $S$ intersects
at least two partite sets, so $G[S]$ is connected. Moreover,
$|S\cap V_1|\leq 1$. Thus, every pair of non-adjacent vertices
in $S$ lies in a partite set distinct from $V_1$ and has $x$
as a common neighbour outside $S$. By Lemma~\ref{P6.lem1},
$S$ is a mutual-visibility set and hence a connected
mutual-visibility set. Therefore, $\mu_c(G)\geq n-1$.
Since $G$ is not complete, Proposition~\ref{P16.prop1}
gives $\mu_c(G)< n$, and consequently $\mu_c(G)=n-1$.

Finally, suppose that $n_1\geq 3$. We first prove that
$\mu_c(G)\leq n-2$. Since $G$ is not complete,
$\mu_c(G)\leq n-1$. Suppose that there exists a connected
mutual-visibility set $S$ of cardinality $n-1$. Let
$V(G)\setminus S=\{x\}$, where $x\in V_j$. Since $n_j\geq 3$, the set
$S\cap V_j$ contains two distinct vertices, say $u$ and $v$. The
vertices $u$ and $v$ are non-adjacent, and all their common neighbours
belong to $V(G)\setminus V_j$. Since $x\in V_j$, every common neighbour
of $u$ and $v$ belongs to $S$. Hence
$(N_G(u)\cap N_G(v))\setminus S=\emptyset$, contradicting
Lemma~\ref{P6.lem1}. Therefore, $\mu_c(G)\leq n-2$. For the reverse inequality, choose vertices $x\in V_1$ and $y\in V_2$,
and set $S=V(G)\setminus\{x,y\}$. Since every partite set contains at
least three vertices, $G[S]$ is connected. Let $u,v\in S$ be
non-adjacent. Then $u$ and $v$ belong to the same partite set. If they
belong to $V_1$, then $y$ is a common neighbour of $u$ and $v$ outside
$S$. Otherwise, $x$ is such a common neighbour. By Lemma~\ref{P6.lem1},
$S$ is a mutual-visibility set. Hence $S$ is a connected
mutual-visibility set, so $\mu_c(G)\geq n-2$ and the equality follows.
\end{proof}
\begin{coro}\label{P16.coro2}
For integers $m,n\geq 2$, $
\mu_c(K_{m,n})=\max\{m+n-2,m+1,n+1\}$.
\end{coro}

\begin{proof}
Let $a=\min\{m,n\}$ and $b=\max\{m,n\}$. Applying
Theorem~\ref{P16.th2} with $r=2$, $n_1=a$, and $n_2=b$, we obtain
\[
\mu_c(K_{m,n})=
\begin{cases}
m+n-1, & a=2,\\[1mm]
m+n-2, & a\geq 3.
\end{cases}
\]
Since $m+n=a+b$ and $\max\{m+1,n+1\}=b+1$, we have
\[
\max\{m+n-2,m+1,n+1\}
=\max\{a+b-2,b+1\}.
\]
If $a=2$, then this maximum is $b+1=m+n-1$. If $a\geq 3$, then
$b+1\leq a+b-2$, and hence the maximum is $a+b-2=m+n-2$. Therefore,
the claimed formula follows.
\end{proof}
\begin{theorem}\label{P16.th4}
Let $G$ be a $(d,2,-\delta)$-graph, and let $S$ be a connected
mutual-visibility set. Then
\[
\begin{aligned}
|S|(|S|-d-1)
&\leq
\sum_{x\in V(G)\setminus S}
|N_G(x)\cap S|
\bigl(|N_G(x)\cap S|-2\bigr)\\
&\leq
(d-2)\min\bigl\{
d(d^2+1-\delta-|S|),(d-2)|S|+2
\bigr\}.
\end{aligned}
\]
\end{theorem}
\begin{proof}
Let $q$ denote the number of unordered non-adjacent pairs of vertices
of $S$. Since $S$ is a mutual-visibility set and $G$ has diameter~$2$,
Lemma~\ref{P6.lem1} implies that every non-adjacent pair of vertices
of $S$ has a common neighbour in $S^c =
V(G)\setminus S$. Hence,
\[
q\le
\sum_{x\in S^c}
\binom{|N_G(x)\cap S|}{2}.
\]

Since $G$ is $d$-regular, $d|S|=2|E(G[S])|+e(S,S^c)$, and
$q=\binom{|S|}{2}-|E(G[S])|$. Therefore,
$2q=|S|(|S|-1)-2|E(G[S])|
=|S|(|S|-d-1)+e(S,S^c)$.
Combining these identities gives
\[
|S|(|S|-d-1)+e(S,S^c)
\le
\sum_{x\in S^c}
|N_G(x)\cap S|
\bigl(|N_G(x)\cap S|-1\bigr).
\]
Since
$e(S,S^c)=\sum_{x\in S^c}|N_G(x)\cap S|$,
subtracting $e(S,S^c)$ from both sides yields
\[
|S|(|S|-d-1)
\le
\sum_{x\in S^c}
|N_G(x)\cap S|
\bigl(|N_G(x)\cap S|-2\bigr).
\]

For the upper bound, since
$0\le |N_G(x)\cap S|\le d$ for every
$x\in S^c$, we have
$|N_G(x)\cap S|\bigl(|N_G(x)\cap S|-2\bigr)
\leq |N_G(x)\cap S|(d-2)$.
Summing over all vertices in $V(G)\setminus S$ gives
\begin{equation}\label{P16.eq1}
\sum_{x\in S^c}
|N_G(x)\cap S|
\bigl(|N_G(x)\cap S|-2\bigr)
\le
(d-2)e(S,S^c).
\end{equation}

Since $G$ has order $d^2+1-\delta$, we have
$|S^c|=d^2+1-\delta-|S|$, and hence
$e(S,S^c)\le d(d^2+1-\delta-|S|)$.
Moreover, $G[S]$ is connected, so
$|E(G[S])|\ge |S|-1$. Using
$e(S,S^c)=d|S|-2|E(G[S])|$, we obtain
$e(S,S^c)\le(d-2)|S|+2$. Therefore,
\[
e(S,S^c)
\le
\min\bigl\{
d(d^2+1-\delta-|S|),
(d-2)|S|+2
\bigr\}.
\]
Substituting this bound into~\eqref{P16.eq1} completes the proof.
\end{proof}

The unique $(d,k,-1)$-graph is the cycle $C_{2k}$, where $d=2$ and
$k\ge2$~\cite{ErdSieHof1980,BanIto1981}. For defect $2$ and diameter
$2$, four graphs are currently known. Up to isomorphism, there are
exactly two $(3,2,-2)$-graphs of order $8$~\cite{Jorgensen1992},
illustrated in Figure~\ref{P6.fig1}. One of them is triangle-free,
whereas the other contains a triangle. We next determine their
connected mutual-visibility number.
\begin{center}
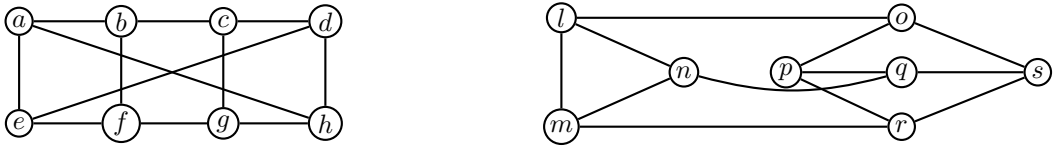

\begin{minipage}{0.4\textwidth}
\centering
\begin{tikzpicture}[
    scale=0.9,line width=0.8pt,
    every node/.style={circle, draw, fill=white, inner sep=1.2pt, font=\small}, minimum size=3.5mm,
    >=stealth
]

\node (a) at (0,1.5) {$a$};
\node (b) at (1.5,1.5) {$b$};
\node (c) at (3,1.5) {$c$};
\node (d) at (4.5,1.5) {$d$};
\node (e) at (0,0) {$e$};
\node (f) at (1.5,0) {$f$};
\node (g) at (3,0) {$g$};
\node (h) at (4.5,0) {$h$};
\draw (a)--(b)--(c)--(d);
\draw (e)--(f)--(g)--(h);

\draw (a)--(e);
\draw (b)--(f);
\draw (c)--(g);
\draw (d)--(h);

\draw (a)--(h);
\draw (e)--(d);

\end{tikzpicture}
\end{minipage}
\hfill
\begin{minipage}{0.5\textwidth}
\centering
\begin{tikzpicture}[
    scale=0.9,line width=0.8pt,
    every node/.style={circle, draw, fill=white, inner sep=1.2pt, font=\small}, minimum size=3.5mm,
    >=stealth
]
\node (a) at (0,1.6) {$l$};
\node (b) at (0,0)   {$m$};
\node (c) at (1.8,0.8) {$n$};
\node (d) at (5.0,1.6) {$o$};
\node (e) at (3.3,0.8) {$p$};
\node (f) at (5.0,0.8) {$q$};
\node (g) at (5.0,0)   {$r$};
\node (h) at (7.0,0.8) {$s$};

\draw (a)--(b)--(c)--(a);

\draw (a)--(d);
\draw (b)--(g);
\draw[bend right=15] (c) to (f);

\draw (d)--(e)--(f);
\draw (e)--(g);
\draw (d)--(h);
\draw (f)--(h);
\draw (g)--(h);
\end{tikzpicture}
\end{minipage}
\captionof{figure}{The two non-isomorphic  $(3,2,-2)$-graphs}\label{P6.fig1}
\end{center}
\begin{theorem}\label{P16.th5}
Let $G$ be a $(3,2,-2)$-graph. Then $\mu_c(G)=4$.
\end{theorem}
\begin{proof}
Let $S$ be a connected mutual-visibility set of $G$. Applying
Theorem~\ref{P16.th4} with $d=3$ and $\delta=2$ gives
$|S|(|S|-4)\leq\min\{3(8-|S|),|S|+2\}$.
If $|S|\ge6$, then
\[
|S|(|S|-4)\ge2|S|>|S|+2
\ge\min\{3(8-|S|),|S|+2\},
\]
a contradiction. Hence, $|S|\le5$.

Suppose that $|S|=5$. Since $G$ has order $8$, we have $|S^c|=3$. By
Theorem~\ref{P16.th4},
\[
5\le
\sum_{x\in S^c}
|N_G(x)\cap S|
\bigl(|N_G(x)\cap S|-2\bigr).
\]
Since $G$ is cubic,
$|N_G(x)\cap S|\in\{0,1,2,3\}$ for every $x\in S^c$, and
$|N_G(x)\cap S|\bigl(|N_G(x)\cap S|-2\bigr)$ is positive only when
$|N_G(x)\cap S|=3$, in which case it equals $3$. Therefore, at least
two vertices of $S^c$ have all their neighbours in $S$ and are thus
isolated in $G[S^c]$. As $|S^c|=3$, the set $S^c$ is independent. Hence every vertex of
$S^c$ has three neighbours in $S$, and therefore
$e(S,S^c)=9$. Since $G$ is cubic, $15=2|E(G[S])|+9$, and therefore
$|E(G[S])|=3$, contradicting the connectedness of $G[S]$, as every
connected graph of order $5$ has at least $4$ edges. Thus,
$\mu_c(G)\le4$.

Up to isomorphism, there are exactly two $(3,2,-2)$-graphs. We use the labelling in
Figure~\ref{P6.fig1}. In the triangle-free graph, the set
$S_1=\{a,b,d,e\}$ induces the path $(b,a,e,d)$. Its non-adjacent pairs
$\{a,d\}$, $\{b,d\}$, and $\{b,e\}$ have common neighbours $h$, $c$,
and $f$, respectively, in $V(G)\setminus S_1$. Hence,
Lemma~\ref{P6.lem1} implies that $S_1$ is a connected
mutual-visibility set.

In the graph containing a triangle, the set
$S_2=\{o,p,q,r\}$ induces a star centred at $p$. Every pair of leaves
has the common neighbour $s\notin S_2$. Hence,
Lemma~\ref{P6.lem1} implies that $S_2$ is also a connected
mutual-visibility set.

Thus $\mu_c(G)\ge4$ in both cases, and consequently
$\mu_c(G)=4$.
\end{proof}
\subsection{Connected mutual visibility in graph joins}
Recall that the join $G\vee H$ is obtained from the disjoint union of
$G$ and $H$ by adding all edges between $V(G)$ and $V(H)$. Thus every
non-adjacent pair of vertices in $G\vee H$ lies entirely in one of the
two factors. Moreover, ${\rm diam}(G\vee H)\le2$; indeed,
${\rm diam}(G\vee H)=1$ if both $G$ and $H$ are complete, and
${\rm diam}(G\vee H)=2$ otherwise. We now determine the connected
mutual-visibility number of the join.
\begin{theorem}\label{P16.th6}
Let $G$ and $H$ be graphs, not necessarily connected, each of order at least $2$. Then
\[
\mu_c(G\vee H)=
\begin{cases}
|V(G\vee H)|,
 & \text{if $G$ and $H$ are complete},\\[1mm]
|V(G\vee H)|-1,
 & \begin{array}{l}
   \text{if at least one of $G$ and $H$ is not complete and}\\
   \mu(G)\geq |V(G)|-1\text{ or }\mu(H)\geq |V(H)|-1,
   \end{array}\\[3mm]
|V(G\vee H)|-2,
 & \begin{array}{l}
   \text{if }\mu(G)<|V(G)|-1\text{ and} \ \mu(H)<|V(H)|-1.   
   \end{array}
\end{cases}
\]
\end{theorem}
\begin{proof}
If $G$ and $H$ are complete, then $G\vee H$ is complete. Hence,
Proposition~\ref{P16.prop1} gives
$\mu_c(G\vee H)=|V(G\vee H)|$.

Suppose that at least one of $G$ and $H$ is not complete. Then
$G\vee H$ is not complete, and Proposition~\ref{P16.prop1} yields
\begin{equation}\label{P16.eq2}
\mu_c(G\vee H)\le |V(G\vee H)|-1.
\end{equation}

Choose $g\in V(G)$ and $h\in V(H)$, and let
$S=V(G\vee H)\setminus\{g,h\}$. Since both $G$ and $H$ have order at
least $2$, the set $S$ meets both $V(G)$ and $V(H)$, and hence
$(G\vee H)[S]$ is connected. Every non-adjacent pair of vertices in $S$
lies entirely in $V(G)$ or entirely in $V(H)$. If
$u,v\in V(G)\cap S$ are non-adjacent, then $h$ is a common neighbour of
$u$ and $v$ outside $S$. Similarly, every non-adjacent pair in
$V(H)\cap S$ has $g$ as a common neighbour outside $S$. Therefore,
Lemma~\ref{P6.lem1} implies that $S$ is a connected
mutual-visibility set. Consequently,
\begin{equation}\label{P16.eq3}
\mu_c(G\vee H)\ge |V(G\vee H)|-2.
\end{equation}

Suppose now that $\mu(G)\ge |V(G)|-1$ or
$\mu(H)\ge |V(H)|-1$. Without loss of generality, assume that
$\mu(G)\ge |V(G)|-1$. Then there exists $z\in V(G)$ such that
$V(G)\setminus\{z\}$ is a mutual-visibility set of $G$. Indeed, if
$\mu(G)=|V(G)|-1$, choose $z$ outside a $\mu$-set of $G$; if
$\mu(G)=|V(G)|$, any vertex of $G$ may be chosen.

Let $S=V(G\vee H)\setminus\{z\}$. Since $G$ has order at least $2$, the
set $S$ contains vertices from both factors, and hence
$(G\vee H)[S]$ is connected. Let $u,v\in S$ be non-adjacent. If
$u,v\in V(H)$, then $(u,z,v)$ is a geodesic whose internal vertex lies
outside $S$. If $u,v\in V(G)\setminus\{z\}$, then, since
$V(G)\setminus\{z\}$ is a mutual-visibility set of $G$, there is a
$(u,v)$-geodesic in $G$ whose internal vertices lie outside
$V(G)\setminus\{z\}$. As $z$ is the only such vertex, this geodesic is
$(u,z,v)$. Hence, by Lemma~\ref{P6.lem1}, $S$ is a connected
mutual-visibility set of $G\vee H$. Together with~\eqref{P16.eq2}, this
yields
$\mu_c(G\vee H)=|V(G\vee H)|-1$.

Finally, suppose that
$\mu(G)<|V(G)|-1$ and $\mu(H)<|V(H)|-1$. Assume that
$\mu_c(G\vee H)=|V(G\vee H)|-1$, and let $S$ be a connected
mutual-visibility set of this cardinality. Let
$V(G\vee H)\setminus S=\{z\}$.

Suppose that $z\in V(G)$. We show that
$V(G)\setminus\{z\}$ is a mutual-visibility set of $G$. Let
$u,v\in V(G)\setminus\{z\}$ be non-adjacent. Since $u,v\in S$, they are
mutually visible with respect to $S$. As $z$ is the only vertex outside
$S$, the corresponding geodesic is $(u,z,v)$, which is also a geodesic
in $G$. Hence $V(G)\setminus\{z\}$ is a mutual-visibility set of $G$,
contradicting $\mu(G)<|V(G)|-1$. The case $z\in V(H)$ similarly
contradicts $\mu(H)<|V(H)|-1$. Therefore,
$\mu_c(G\vee H)\le |V(G\vee H)|-2$. Together with~\eqref{P16.eq3}, this
gives
$\mu_c(G\vee H)=|V(G\vee H)|-2$.
\end{proof}
Di Stefano~\cite{Stefano} established the same
three-case formula for the classical mutual-visibility number of the
join. Thus, Theorem~\ref{P16.th6} shows that $\mu_c(G\vee H)=\mu(G\vee H)
$ whenever both $G$ and $H$ have order at least $2$.
\begin{prop}\label{P16.prop3}
Let $H$ be a graph. Then
\[
\mu_c(K_1\vee H)=
\begin{cases}
|V(H)|+1, & \text{if $H$ is complete},\\[1mm]
|V(H)|,   & \text{if $H$ is not complete}.
\end{cases}
\]
\end{prop}
\begin{proof}
If $H$ is complete, then $K_1\vee H$ is complete. Hence,
Proposition~\ref{P16.prop1} gives
$\mu_c(K_1\vee H)=|V(K_1\vee H)|=|V(H)|+1$.

Suppose that $H$ is not complete, and let $z$ denote the vertex of
$K_1$. Put $S=V(H)$. Since $H$ is connected, $(K_1\vee H)[S]=H$ is
connected. If $u,v\in S$ are adjacent, then $uv$ is a geodesic. If they
are non-adjacent, then $(u,z,v)$ is a geodesic whose internal vertex
lies outside $S$. Hence $S$ is a connected mutual-visibility set, and
therefore $\mu_c(K_1\vee H)\ge |V(H)|$.

Since $H$ is not complete, $K_1\vee H$ is not complete. Thus,
Proposition~\ref{P16.prop1} yields
$\mu_c(K_1\vee H)<|V(K_1\vee H)|=|V(H)|+1$, and hence
$\mu_c(K_1\vee H)\le |V(H)|$. Therefore,
$\mu_c(K_1\vee H)=|V(H)|$.
\end{proof}
\section{Some Consequences of Main Results}
In this section, we derive several consequences of the main results by determining the connected mutual-visibility number of block graphs, cycles, cactus graphs, and corona products. These consequences illustrate how geodeticity, girth, and block structure influence connected mutual visibility.

\begin{prop}
If $G$ is a block graph, then
$\mu_c(G)=\omega(G)$.
\end{prop}

\begin{proof}
Every connected block graph is geodetic. Therefore, the result follows
immediately from Theorem~\ref{P16.th1}.
\end{proof}
\begin{prop}\label{P16.prop4}
For every $n\geq 3$,
$$\mu_c(C_n)=
\begin{cases}
3, & n=3,4\\
2, & n\geq 5.
\end{cases}$$
\end{prop}
\begin{proof}
The case $n=3$ is immediate, since $C_3\cong K_3$. Since
$C_4\cong K_{2,2}$, Corollary~\ref{P16.coro2} gives
$\mu_c(C_4)=3$. For $n\geq 5$, the cycle $C_n$ has girth $n$, and hence
Theorem~\ref{P16.th3} yields $\mu_c(C_n)=2$.
\end{proof}
\begin{prop}\label{P16.cor3}
Let $G$ be a triangle-free graph of order at least $2$. Then
$\mu_c(G)\le c_4(G)+2$.
Moreover, the bound is sharp.
\end{prop}
\begin{proof}
Since $G$ is triangle-free, $c_3(G)=0$, and Theorem~\ref{P16.th9}
yields $\mu_c(G)\le c_4(G)+2$. The bound is sharp, since
$c_4(C_4)=1$ and $\mu_c(C_4)=3$ by Proposition~\ref{P16.prop4}.
\end{proof}
The sharpness assertion can be strengthened considerably. For every
non-negative integer $q$, there exists a connected triangle-free graph
$G_q$ with $c_4(G_q)=q$ and $\mu_c(G_q)=q+2$. Indeed, let
$P=(v_0,v_1,\ldots,v_{q+1})$ be a path, and for every pair
$v_i,v_j$ with $j-i\ge2$, add a new vertex $x_{ij}$ adjacent only to
$v_i$ and $v_j$. The resulting graph is triangle-free, and its only
$4$-cycles are
$(v_i,v_{i+1},v_{i+2},x_{i(i+2)},v_i)$ for
$0\le i\le q-1$. Hence $c_4(G_q)=q$. Now let
$S=\{v_0,v_1,\ldots,v_{q+1}\}$. Then $G_q[S]=P$ is connected, and for
every pair $v_i,v_j\in S$ with $j-i\ge2$, the path
$(v_i,x_{ij},v_j)$ is a $(v_i,v_j)$-geodesic whose unique internal
vertex lies outside $S$. Thus $S$ is a connected mutual-visibility
set, so $\mu_c(G_q)\ge q+2$. The reverse inequality follows from
Proposition~\ref{P16.cor3}. Therefore,
$\mu_c(G_q)=q+2=c_4(G_q)+2$.
\begin{prop}\label{P16.prop2}
Let $G$ be a cactus graph of order at least $2$. Then
\[
\mu_c(G)=
\begin{cases}
3, & \text{if $G$ has a block isomorphic to $C_3$ or $C_4$},\\
2, & \text{otherwise}.
\end{cases}
\]
\end{prop}
\begin{proof}
Every block of a cactus graph is isomorphic to either $K_2$ or a cycle.
Moreover, by Proposition~\ref{P16.prop1},
$\mu_c(K_2)=2$; by Proposition~\ref{P16.prop2},
$\mu_c(C_3)=\mu_c(C_4)=3$ and $\mu_c(C_n)=2$ for every $n\ge5$.
The result now follows from Theorem~\ref{P16.th7}.
\end{proof}
\begin{theorem}\label{P16.th8}
Let $G$ and $H$ be graphs, where $|V(G)|\geq2$. Then
\[
\mu_c(G\circ H)=
\begin{cases}
\max\{\mu_c(G),|V(H)|+1\}, & \text{if $H$ is complete},\\[1mm]
\max\{\mu_c(G),|V(H)|\},   & \text{if $H$ is not complete}.
\end{cases}
\]
\end{theorem}
\begin{proof}
For each $v\in V(G)$, let $H_v$ denote the copy of $H$ associated with
$v$ in $G\circ H$, and let
$B_v=(G\circ H)[\{v\}\cup V(H_v)]$. Then
$B_v\cong K_1\vee H$. By Proposition~\ref{P16.prop3},
$B_v$ is complete if and only if $H$ is complete.

We claim that each $B_v$ is a block of $G\circ H$. If $|V(H)|\ge2$,
then deleting $v$ leaves the connected graph $H_v$, while deleting a
vertex of $H_v$ leaves $v$ adjacent to every remaining vertex.
Therefore, $B_v$ has no cut-vertex. Moreover, $v$ is the only vertex
through which $H_v$ is adjacent to the rest of $G\circ H$, so $B_v$ is
maximal with this property. If $H\cong K_1$, then
$B_v\cong K_2$, which is also a block.

The remaining blocks of $G\circ H$ are precisely those inherited from
$G$. Hence, by Theorem~\ref{P16.th7},
$\mu_c(G\circ H)
=\max\{\mu_c(G),\mu_c(K_1\vee H)\}$.
The result now follows from Proposition~\ref{P16.prop3}.
\end{proof}
\section{Computational Complexity}
In this section, we investigate the computational complexity of
determining the connected mutual-visibility number. We begin by
formulating the corresponding decision problem.

\begin{defn}
\textsc{Connected Mutual-Visibility} problem:

\leavevmode\par
\textit{Instance:} A graph $G$ and a positive integer
$k\le |V(G)|$.

\textit{Question:} Is $\mu_c(G)\ge k$?
\end{defn}

We next investigate the computational complexity of
\textsc{Connected Mutual-Visibility}. To establish
$\mathsf{NP}$-completeness, we present a polynomial-time reduction from
the classical \textsc{Clique} problem, whose definition is recalled
below.

\begin{quote}
\noindent
\textit{Instance:} A graph $F$ and a positive integer
$k\le |V(F)|$.

\noindent
\textit
{Question:} Is $\omega(F)\ge k$?
\end{quote}
\begin{theorem}
The \textsc{Connected Mutual-Visibility} is
$\mathsf{NP}$-complete. Moreover, it remains
$\mathsf{NP}$-complete when restricted to connected bipartite graphs
of diameter at most~$4$.
\end{theorem}
\begin{proof}
Let $S\subseteq V(G)$. Whether $S$ is a mutual-visibility set of $G$
can be verified in $O(|S|(|V(G)|+|E(G)|))$ time
\cite{Stefano}. The connectedness of the induced subgraph $G[S]$ can
be checked by a breadth-first search in
$O(|S|+|E(G[S])|)$ time. Hence, whether $S$ is a connected
mutual-visibility set of $G$ can be determined in
$O\bigl(|S|(|V(G)|+|E(G)|)\bigr)$ time. Therefore, the
\textsc{Connected Mutual-Visibility} problem belongs to
$\mathsf{NP}$.

To prove $\mathsf{NP}$-hardness, we give a polynomial-time reduction
from \textsc{Clique}, which is known to be $\mathsf{NP}$-complete \cite{Karp1972}. Let $(F,k)$ be an instance of
\textsc{Clique}. Construct a graph $\widehat{F}$ from $F$ by adding a
new vertex $c$ and, for each edge $uv\in E(F)$, a new vertex
$x_{uv}$. Let $W=\{x_{uv}:uv\in E(F)\}$. Then
$V(\widehat{F})=V(F)\cup W\cup\{c\}$ and
$E(\widehat{F})=\{cv:v\in V(F)\}\cup
\{ux_{uv},vx_{uv}:uv\in E(F)\}$. A representation of $\widehat{F}$ is depicted in Figure~\ref{P16.fig2}.
\begin{center}
\begin{minipage}{0.4\textwidth}
\centering
\begin{tikzpicture}[
    scale=1,
    vertex/.style={circle, draw, fill=white, minimum size=6mm,
                   inner sep=0pt},
    newvertex/.style={circle, draw, fill=red!65, minimum size=6mm,
                      inner sep=0pt}
]

\node[newvertex] (u) at (0,1) {$u$};
\node[newvertex] (v) at (2.4,1) {$v$};
\node[newvertex] (w) at (1.2,-0.4) {$w$};
\node[vertex] (z) at (1.2,-1.8) {$z$};

\draw (u)--(v);
\draw (u)--(w);
\draw (v)--(w);
\draw (w)--(z);
\end{tikzpicture}

\smallskip
The graph $F$
\end{minipage}
\hfill
\begin {minipage}{0.5\textwidth}
\centering
\begin{tikzpicture}[
    scale=0.9,
    vertex/.style={circle, draw, fill=white, minimum size=6mm,
                   inner sep=0pt},
    newvertex/.style={circle, draw, fill=red!65, minimum size=6mm,
                      inner sep=0pt}
]
\node[newvertex] (c) at (3,2.4) {$c$};

\node[newvertex] (U) at (0,0) {$u$};
\node[newvertex] (V) at (2,0) {$v$};
\node[vertex] (W) at (4,0) {$w$};
\node[vertex] (Z) at (6,0) {$z$};

\node[vertex] (xuv) at (0.5,-2) {$x_{uv}$};
\node[newvertex] (xuw) at (2.2,-2) {$x_{uw}$};
\node[vertex] (xvw) at (3.8,-2) {$x_{vw}$};
\node[vertex] (xwz) at (5.5,-2) {$x_{wz}$};

\draw (c)--(U);
\draw (c)--(V);
\draw (c)--(W);
\draw (c)--(Z);

\draw (xuv)--(U);
\draw (xuv)--(V);

\draw (xuw)--(U);
\draw (xuw)--(W);

\draw (xvw)--(V);
\draw (xvw)--(W);

\draw (xwz)--(W);
\draw (xwz)--(Z);
\end{tikzpicture}

\smallskip
The graph $\widehat{F}$
\end{minipage}

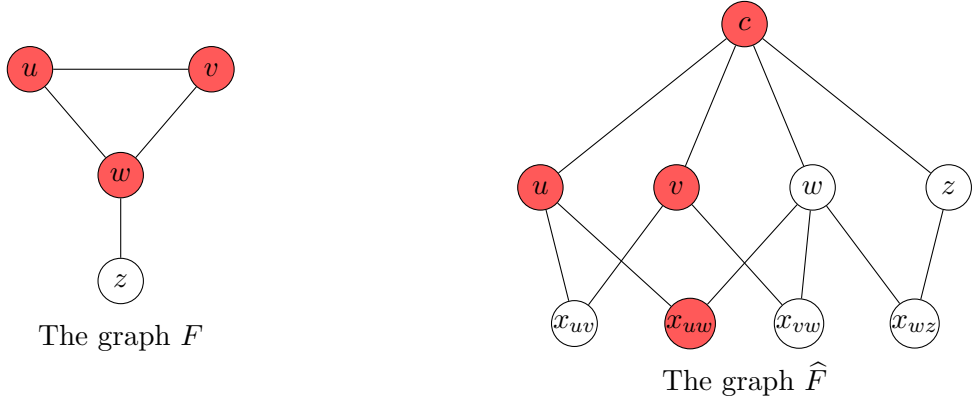
\captionof{figure}{The construction of $\widehat F$ from $F$. The
coloured vertices indicate a maximum clique of $F$ and a $\mu_c$-set
of $\widehat F$, respectively. (For interpretation of the references to colour
in this figure legend, the reader is referred to the web version of
this article.)}\label{P16.fig2}
\end{center}

The construction is clearly polynomial-time. Moreover,
$\widehat{F}$ is bipartite with bipartition
$\bigl(V(F),W\cup\{c\}\bigr)$, and every two vertices of
$\widehat{F}$ are at distance at most~$4$. Hence
${\rm diam}(\widehat{F})\le4$.

We claim that
\begin{equation}
\mu_c(\widehat{F})=\omega(F)+1.
\label{eq:cmv-clique}
\end{equation}

Let $C$ be a maximum clique of $F$, and let
$S=C\cup\{c\}$. Then $\widehat{F}[S]$ is a star centred at $c$, and
hence is connected. It remains to verify that $S$ is a
mutual-visibility set. Since $c$ is adjacent to every vertex of $C$,
every pair involving $c$ is trivially $S$-visible. Now let
$u,v\in C$ be distinct. Since $C$ is a clique of $F$,
$uv\in E(F)$, and therefore
$x_{uv}\in V(\widehat{F})\setminus S$. Moreover,
$(u,x_{uv},v)$ is a $(u,v)$-geodesic in $\widehat{F}$ whose unique
internal vertex lies outside $S$. Hence $u$ and $v$ are
$S$-visible. Therefore, $S$ is a connected mutual-visibility set of
$\widehat{F}$, and consequently
$\mu_c(\widehat{F})\ge |S|=\omega(F)+1$.

Let $S$ be a $\mu_c$-set of \(\widehat{F}\).

\noindent
\textbf{Case 1.} Suppose that $c\notin S$. If $S$ contains at least
two vertices of $W$, then, since $\widehat{F}[S]$ is connected, a path
in $\widehat{F}[S]$ joining two vertices of $W$ contains a subpath
$(x_e,u,x_f)$, where $x_e,x_f\in W\cap S$ are distinct and
$u\in V(F)\cap S$. The vertices $x_e$ and $x_f$ are at distance $2$ in
$\widehat{F}$. Moreover, since $F$ is simple, the distinct edges $e$
and $f$ have at most one common endpoint. Hence $u$ is their unique
common neighbour, so $(x_e,u,x_f)$ is the unique
$(x_e, x_f)$-geodesic in $\widehat{F}$. Its internal vertex belongs to
$S$, contradicting the mutual visibility of $x_e$ and $x_f$.
Therefore, $|S\cap W|\le1$.

If $S\cap W=\emptyset$, then $S\subseteq V(F)$. Since $V(F)$ is an
independent set in $\widehat{F}$, the connectedness of
$\widehat{F}[S]$ implies that $|S|\le1$. Now suppose that
$S\cap W=\{x_{uv}\}$ for some $uv\in E(F)$. Then every vertex of
$S\cap V(F)$ is adjacent to $x_{uv}$. Since $N_{\widehat{F}}(x_{uv})=\{u,v\}$, it follows that
$S\subseteq\{u,x_{uv},v\}$, and hence $|S|\le3$.  Since
$uv\in E(F)$, we have $\omega(F)\ge2$. Therefore,
$|S|\le3\le\omega(F)+1$.

\noindent
\textbf{Case 2.} Suppose that $c\in S$. Put $F'=S\cap V(F)$ and $W'=S\cap W$.
Since $\{c,v\}$ is a connected mutual-visibility set of $\widehat F$
for every $v\in V(F)$, we have $|S|\ge2$. Hence, as
$\widehat F[S]$ is connected and $N_{\widehat F}(c)=V(F)$, it follows
that $F'\neq\emptyset$.

We show that $F'$ is a clique of $F$. If $|F'|=1$, then this is
immediate. Suppose that $|F'|\geq2$, and let $u,v\in F'$ be distinct. In $\widehat F$, the vertices $u$ and $v$ have the common
neighbour $c$, and, if $uv\in E(F)$, also the common neighbour
$x_{uv}$. Since $c\in S$, the mutual visibility of $u$ and $v$ implies
that $uv\in E(F)$ and $x_{uv}\notin S$. Hence every two distinct
vertices of $F'$ are adjacent in $F$, and so $F'$ is a clique.

For each $x_{uv}\in W'$, the connectedness of $\widehat F[S]$ implies
that at least one of $u$ and $v$ belongs to $F'$. On the other hand,
$c$ and $x_{uv}$ are non-adjacent, and their only geodesics are
$(c,u,x_{uv})$ and $(c,v,x_{uv})$. Since $c$ and $x_{uv}$ are
$S$-visible, at least one of $u$ and $v$ lies outside $F'$.
Consequently, exactly one endpoint of $uv$ belongs to $F'$. We may therefore define maps $p:W'\longrightarrow F'$ and
$q:W'\longrightarrow V(F)\setminus F'$ as follows. For every
$x_{uv}\in W'$, let $p(x_{uv})$ be the unique endpoint of $uv$ in
$F'$, and let $q(x_{uv})$ be the other endpoint.

Next we show that $q$ is injective. Suppose not. Then there exist distinct vertices
$x_e,x_f\in W'$ such that $q(x_e)=q(x_f)$. Let
$q(x_e)=q(x_f)=v$, $p(x_e)=u$, and $p(x_f)=w$. Since $F$ is simple and
$x_e\neq x_f$, we have $u\neq w$. Thus $x_e=x_{uv}$ and
$x_f=x_{wv}$. Consider the vertices $x_{uv}$ and $w$, both of which belong to $S$.
They are at distance $3$ in $\widehat{F}$. Every
$(x_{uv},w)$-geodesic begins through either $u$ or $v$. If it begins
through $u$, then it has the internal vertex $u\in F'\subseteq S$. If
it begins through $v$, then it must next pass through either $c$ or
$x_{wv}$, both of which belong to $S$. Thus every
$(x_{uv},w)$-geodesic contains an internal vertex of $S$,
contradicting the mutual visibility of $x_{uv}$ and $w$. Therefore,
$q$ is injective.

Let $Q=\{q(x):x\in W'\}$. We claim that $F'\cup Q$ is a clique of
$F$. Let $a\in F'$ and
$b\in Q$. By the definition of $Q$, there exists $x\in W'$ such that
$b=q(x)$. If $a=p(x)$, then $ab\in E(F)$. Suppose that
$a\neq p(x)$. Then $x$ and $a$ are at distance $3$ in $\widehat{F}$.
Since $x$ and $a$ are $S$-visible, there is an $(x,a)$-geodesic whose
internal vertices lie outside $S$. This geodesic must begin with
$xb$, since the other neighbour $p(x)$ of $x$ belongs to $S$. Its next
internal vertex cannot be $c$, because $c\in S$. Hence it must be
$x_{ab}$, which exists only if $ab\in E(F)$. Therefore every vertex of
$F'$ is adjacent in $F$ to every vertex of $Q$.

Finally, we show that $Q$ is a clique of $F$. If $|Q|\leq1$, then
this is immediate. Suppose that $|Q|\geq2$, and let $b,d\in Q$ be
distinct. Choose $x,y\in W'$ such that
$b=q(x)$ and $d=q(y)$. Suppose that $p(x)=p(y)$. Since $x$ and $y$
correspond to distinct edges of the simple graph $F$, their common
endpoint $p(x)=p(y)$ is the unique common neighbour of $x$ and $y$ in
$\widehat F$, and it belongs to $S$, contradicting the mutual
visibility of $x$ and $y$. Hence $p(x)\neq p(y)$. Since $b\neq d$, $p(x)\neq p(y)$, and $F'\cap Q=\emptyset$, the edges
of $F$ corresponding to $x$ and $y$ have no common endpoint. Hence
$d_{\widehat{F}}(x,y)=4$. Since $x$ and $y$ are $S$-visible, there is
an $(x,y)$-geodesic whose internal vertices lie outside $S$. This
geodesic must begin with $xb$, since $p(x)\in S$, and must end with
$dy$, since $p(y)\in S$. Its middle vertex cannot be $c$, because
$c\in S$. Hence it must be $x_{bd}$, which exists only if
$bd\in E(F)$. Thus every two distinct vertices of $Q$ are adjacent in
$F$. Consequently, $F'\cup Q$ is a clique of $F$.

Since $q$ is injective and $F'\cap Q=\emptyset$, we obtain $|S|-1=|F'|+|W'|=|F'|+|Q|=|F'\cup Q|\le\omega(F)
$. Therefore, $|S|\le\omega(F)+1$. Together with the preceding lower
bound, this proves~\eqref{eq:cmv-clique}.

Consequently,
\[
\omega(F)\ge k
\quad\Longleftrightarrow\quad
\mu_c(\widehat{F})\ge k+1.
\]
Thus, $(F,k)$ is a yes-instance of \textsc{Clique} if and only if
$(\widehat{F},k+1)$ is a yes-instance of
\textsc{Connected Mutual-Visibility}. Since the construction is
polynomial-time, \textsc{Connected Mutual-Visibility} is
$\mathsf{NP}$-hard. As it belongs to $\mathsf{NP}$, the problem is
$\mathsf{NP}$-complete.
\end{proof}

\section{Conclusion}

In this paper, we introduced the connected mutual-visibility number
$\mu_c(G)$ and established its fundamental properties. We
characterised its extremal values, proved that it is local with
respect to the block structure, determined its exact value for several
graph classes and graph operations, established
Nordhaus--Gaddum-type inequalities, and investigated regular graphs
of diameter $2$ with small defect. These results demonstrate that the
connectivity requirement may significantly reduce the classical
mutual-visibility number.

From an algorithmic perspective, we showed that recognising connected
mutual-visibility sets can be performed in polynomial time, whereas
the associated decision problem is $\mathsf{NP}$-complete, even when
restricted to connected bipartite graphs of diameter at most $4$. This
demonstrates that determining the connected mutual-visibility number
remains computationally intractable even under strong structural
restrictions.

Several questions remain open. It would be interesting to characterise
the graphs for which $\mu_c(G)=\mu(G)$ and to determine $\mu_c$ for
further structured graph classes, such as chordal and planar graphs.
Another natural direction is to investigate connected mutual
visibility in extremal graphs related to the Moore bound. In
particular, the present study of Moore graphs and graphs with small
defect may be extended to graphs with larger defect and, in the
complementary direction, to graphs with excess,  relative to the
corresponding Moore bound. Determining how defect and excess
influence the structure and cardinality of connected
mutual-visibility sets appears to be a promising direction for future
research.
\bibliographystyle{plainurl}
\bibliography{cas-refs}

\begin{thebibliography}{10}

\bibitem{robotics7}
A.~Aljohani, P.~Poudel, and G.~Sharma.
\newblock Complete visitability for autonomous robots on graphs.
\newblock In {\em Proceedings of the IEEE International Parallel and
  Distributed Processing Symposium (IPDPS)}, pages 733--742, 2018.
\newblock \href {https://doi.org/10.1109/IPDPS.2018.00083}
  {\path{doi:10.1109/IPDPS.2018.00083}}.

\bibitem{robotics4}
R.~J. Alsaedi, J.~Gudmundsson, and A.~van Renssen.
\newblock The mutual visibility problem for fat robots.
\newblock {\em Algorithms and Data Structures. WADS 2023}, 14079:15--28, 2023.
\newblock \href {https://doi.org/10.1007/978-3-031-38906-1_2}
  {\path{doi:10.1007/978-3-031-38906-1_2}}.

\bibitem{BanIto1973}
E~Bannai and T~Ito.
\newblock On finite {Moore} graphs.
\newblock {\em Journal of the Faculty of Science, University of Tokyo, Section
  IA, Mathematics}, 20:191--208, 1973.

\bibitem{BanIto1981}
Eiichi Bannai and Tatsuro Ito.
\newblock Regular graphs with excess one.
\newblock {\em Discrete Mathematics}, 37:147--158, 1981.
\newblock \href {https://doi.org/10.1016/0012-365X(81)90215-6}
  {\path{doi:10.1016/0012-365X(81)90215-6}}.

\bibitem{robotics3}
S.~Bhagat, P.~Dey, and R.~Ray.
\newblock Collision-free linear time mutual visibility for asynchronous fat
  robots.
\newblock In {\em Proceedings of the 25th International Conference on
  Distributed Computing and Networking}, pages 84--93, 2024.
\newblock \href {https://doi.org/10.1145/3631461.3631544}
  {\path{doi:10.1145/3631461.3631544}}.

\bibitem{robotics5}
S.~Cicerone, A.~Di~Fonso, G.~Di~Stefano, and A.~Navarra.
\newblock The geodesic mutual visibility problem for oblivious robots: The case
  of trees.
\newblock In {\em ICDCN '23: Proceedings of the 24th International Conference
  on Distributed Computing and Networking}, pages 150--159, 2023.
\newblock \href {https://doi.org/10.1145/3571306.3571401}
  {\path{doi:10.1145/3571306.3571401}}.

\bibitem{robotics1}
S.~Cicerone, A.~Di~Fonso, G.~Di~Stefano, and A.~Navarra.
\newblock Time-optimal geodesic mutual visibility of robots on grids within
  minimum area.
\newblock In {\em Stabilization Safety and Security of Distributed Systems(SSS
  2023)}, volume 14310, page 385–399, 2023.
\newblock \href {https://doi.org/10.1007/978-3-031-44274-2_29}
  {\path{doi:10.1007/978-3-031-44274-2_29}}.

\bibitem{Dam1973}
R.~M. Damerell.
\newblock On {Moore} graphs.
\newblock {\em Mathematical Proceedings of the Cambridge Philosophical
  Society}, 74:227--236, 1973.

\bibitem{DetMag2025}
Magda Dettlaff, Magdalena Lema\'{n}ska, Juan~Alberto Rodriguez-Velazquez, and
  Ismael~Gonz\'{a}lez Yero.
\newblock Mobile mutual-visibility sets in graphs.
\newblock {\em Ars Mathematica Contemporanea}, 26(1), 2026.
\newblock \href {https://doi.org/10.26493/1855-3974.3410.9bc}
  {\path{doi:10.26493/1855-3974.3410.9bc}}.

\bibitem{robotics2}
G.A. Di~Luna, P.~Flocchini, S.~G. Chaudhuri, F.~Poloni, N.~Santoro, and
  G.~Viglietta.
\newblock Mutual visibility by luminous robots without collisions.
\newblock {\em Information and Computation}, 254:392--418, 2017.

\bibitem{Stefano}
G.~Di~Stefano.
\newblock Mutual visibility in graphs.
\newblock {\em Applied Mathematics and Computation}, 419, 2022.
\newblock \href {https://doi.org/10.1016/j.amc.2021.126850}
  {\path{doi:10.1016/j.amc.2021.126850}}.

\bibitem{ErdSieHof1980}
Paul Erd{\H{o}}s, Siemion Fajtlowicz, and Alan~J. Hoffman.
\newblock Maximum degree in graphs of diameter 2.
\newblock {\em Networks}, 10(1):87--90, 1980.
\newblock \href {https://doi.org/10.1002/net.3230100109}
  {\path{doi:10.1002/net.3230100109}}.

\bibitem{hoffman_moore}
A.~J. Hoffman and R.~R. Singleton.
\newblock On {Moore} graphs with diameter 2 and 3.
\newblock {\em IBM Journal of Research and Development}, 4(5):497--504, 1960.

\bibitem{Jorgensen1992}
Leif~K. J{\o}rgensen.
\newblock Diameters of cubic graphs.
\newblock {\em Discrete Applied Mathematics}, 37/38:347--351, 1992.
\newblock \href {https://doi.org/10.1016/0166-218X(92)90144-Y}
  {\path{doi:10.1016/0166-218X(92)90144-Y}}.

\bibitem{Karp1972}
R.M. Karp.
\newblock Reducibility among combinatorial problems.
\newblock In R.~E. Miller and J.~W. Thatcher, editors, {\em Proceedings of
  Complexity of Computer Computations}, pages 85--103, 1972.

\bibitem{GP_1}
P.~Manuel and S.~Klavžar.
\newblock A general position problem in graph theory.
\newblock {\em Bulletin of the Australian Mathematical Society}, 98:177--187,
  2018.
\newblock \href {https://doi.org/10.1017/S0004972718000473}
  {\path{doi:10.1017/S0004972718000473}}.

\bibitem{MillerSiran2013}
Mirka Miller and Jozef {\v{S}}ir{\'a}{\v{n}}.
\newblock Moore graphs and beyond: A survey of the degree/diameter problem.
\newblock {\em Electronic Journal of Combinatorics}, 20(2):\#DS14v2, 2013.

\bibitem{robotics6}
G.~Sharma.
\newblock Mutual visibility for robots with lights tolerating light faults.
\newblock In {\em Proceedings of the 2018 IEEE International Parallel and
  Distributed Processing Symposium Workshops}, pages 829--836, 2018.
\newblock \href {https://doi.org/10.1109/IPDPSW.2018.00130}
  {\path{doi:10.1109/IPDPSW.2018.00130}}.

\bibitem{MV_3}
J.~Tian and S.~Klavžar.
\newblock Graphs with total mutual-visibility number zero and total
  mutual-visibility in cartesian products.
\newblock {\em Discussiones Mathematicae Graph Theory}, 44:1277--1291, 2024.
\newblock \href {https://doi.org/10.7151/dmgt.2496}
  {\path{doi:10.7151/dmgt.2496}}.

\bibitem{TonShi_dif2026}
K~B. Tonny and M.~Shikhi.
\newblock Mutual visibility in moore graphs and (d,2)-graphs with defect.
\newblock {\em Discrete Mathematics}, 350(1):115385, 2027.
\newblock \href {https://doi.org/10.1016/j.disc.2026.115385}
  {\path{doi:10.1016/j.disc.2026.115385}}.

\end{thebibliography}
\end{document}